\documentclass[12pt,reqno]{amsart}
\usepackage[margin=1.25in]{geometry}        
\usepackage{graphicx}
\usepackage{amssymb}
\usepackage{dsfont}
\usepackage{aliascnt}
\usepackage{doi}
\usepackage{epstopdf}
\usepackage{tikz}
\usetikzlibrary{3d}
\usetikzlibrary {arrows.meta}

\usepackage{esint}

\newtheorem{theorem}{Theorem}[section]

\newaliascnt{corx}{thmx}

\aliascntresetthe{corx}

\newaliascnt{lemma}{theorem}
\newtheorem{lemma}[lemma]{Lemma}
\aliascntresetthe{lemma}

\newaliascnt{proposition}{theorem}
\newtheorem{proposition}[proposition]{Proposition}
\aliascntresetthe{proposition}

\newaliascnt{corollary}{theorem}
\newtheorem{corollary}[corollary]{Corollary}
\aliascntresetthe{corollary}

\newaliascnt{conjecture}{theorem}
\newtheorem{conjecture}[conjecture]{Conjecture}
\aliascntresetthe{conjecture}

\newaliascnt{question}{theorem}

\aliascntresetthe{question}

\theoremstyle{definition}
\newtheorem*{definition*}{Definition}
\newtheorem*{example*}{Example}
\newtheorem*{remark*}{Remark}
\newtheorem*{remarks*}{Remarks}

\newcommand{\B}{{\mathbb B}}
\newcommand{\e}{\varepsilon}

\newcommand{\R}{{\mathbb R}}

\newcommand{\Rn}{{{\mathbb R}^n}}
\newcommand{\Rnp}{{{\mathbb R}^{n+1}}}
\newcommand{\Sph}{{\mathbb S}}

\newcommand{\xh}{\hat{x}}
\newcommand{\yh}{\hat{y}}

\newcommand{\Z}{{\mathbb Z}}

\DeclareMathOperator{\csch}{csch}

\DeclareMathOperator{\capzero}{Cap_0}
\DeclareMathOperator{\capone}{Cap_1}

\DeclareMathOperator{\capq}{Cap_\mathit{q}}

\title{Riesz energy deformation through insulated strips}

\author{Carrie Clark and Richard S. Laugesen}
\address{}
\email{carrieclark435@gmail.com}
\address{University of Illinois, Urbana IL 61801, USA}
\email{Laugesen@illinois.edu}


\begin{document}

\begin{abstract}
For compact sets in Euclidean space, Riesz energies whose exponents differ by $1$ are shown to arise as the endpoint cases of a one-parameter family of infinite-strip energies as the strip thickness increases from $0$ to $\infty$, under Neumann boundary conditions. An approach is suggested to a capacity conjecture of P\'{o}lya and Szeg\H{o}. 
\end{abstract}

\maketitle

\section{\bf Introduction}

How should one interpolate between pairwise interaction energies having different strength singularities? Can one connect the energies in some physically natural manner? The family of Riesz energies provides the prototypical situation. A compact nonempty set $K \subset \Rn$ has Riesz $q$-energy  
\[
V_q(K) = \min_\mu \int_K \! \int_K \frac{1}{|x-y|^q} \, d\mu(x) d\mu(y) , \qquad 0 < q < n . 
\]
The minimum is taken over probability measures on $K$, meaning $\mu$ is a positive unit Borel measure, and the minimum is attained by an ``equilibrium" measure whose properties depend of course on the exponent $q$. (Intuitively, equilibrium measures are the ones that most effectively spread out the repelling charges $\mu$ on the set $K$.) Note the energy is either positive or $+\infty$. One similarly defines the logarithmic energy $V_{log}(K)$ in terms of the logarithmic kernel $\log 1/|x-y|$. 

In classical electrostatics the exponent is $q=n-2$, but if the same set $K$ is regarded as lying in the higher dimensional space $\Rnp$ then the electrostatic exponent is $q=n-1$. Thus there are natural reasons to consider a fixed set $K$ with different exponents $q$. 

Simply varying the exponent $q$ in order to deform one Riesz energy into another, though, would feel artificial. A more natural parameterized family of interpolating energies is suggested by the preceding electrostatic situation: we can try to reduce the ambient dimension from $n+1$ to $n$ by passing through a family of insulated infinite strips, and hence reduce $q$ to $q-1$. 

In order to state our result, take $G_t(\xh,\yh)$ to be the Riesz-type kernel for points $\xh$ and $\yh$ in the strip of thickness $2t$ that is shown in \autoref{fig:strip}: 
\[
S(t) = \Rn \times (-t,t) \subset \Rnp .
\]
This kernel $G_t$ is defined precisely in \autoref{sec:definitions} below: it behaves like $|\xh-\yh|^{-q}$ when $\xh$ and $\yh$ are close together, and the kernel satisfies a Neumann boundary condition when $\xh$ lies at the boundary and $\yh$ is inside the strip. The energy induced by the kernel is 
\[
E_K(t) = \min_\mu \int_K \! \int_K G_t(\xh,\yh) \, d\mu(\xh) d\mu(\yh) , 
\]
where the minimum is taken over all probability measures on the set $K$. Due to the Neumann boundary condition on the kernel, we may interpret  $E_K(t)$ as the energy of a conductor $K$ that is positioned between two insulated hyperplanes in $\Rnp$ at heights $\pm t$. 

Our main result says that as $t$ increases from $0$ to $\infty$, these strip energies form a one-parameter family connecting the endpoint cases, which we show are the Riesz energies with exponents $q-1$ and $q$. 
\begin{theorem}[Connecting Riesz energies for $K \subset \Rn$] \label{th:main}
Fix $1 \leq q < n$. If $K \subset \Rn$ is compact with finite Riesz $q$-energy then $E_K(t)$ is a $C^1$-smooth function of $t>0$ that interpolates between the $q$- and $(q-1)$-energies:  
\[
\lim_{t \to \infty} E_K(t) = V_q(K) \quad {and} \quad  
\liminf_{t \to 0} t E_K(t) \geq 
\begin{cases}
c_{q-1} V_{q-1}(K) & \text{when $q>1$,} \\
V_{log}(K) & \text{when $q=1$,}
\end{cases}
\]
where the positive constant $c_{q-1}$ is defined in \eqref{cdef} below. 

Equality holds as $t \to 0$ if $K$ is interior $(q-1)$-capacitable in $\Rn$ (as defined in \autoref{sec:energyresults}), meaning that in this case $\lim_{t \to 0} t E_K(t) = c_{q-1} V_{q-1}(K)$ when $q>1$ and the limit equals $V_{log}(K)$ when $q=1$. In particular, equality holds as $t \to 0$ if $K$ is a convex body in $\Rn$. 
\end{theorem}
\autoref{th:main} follows from combining the asymptotic formula for the energy as $t \to \infty$ (\autoref{th:energybigt} or \autoref{co:energybigtKinRn}) with the asymptotic as $t \to 0$ (\autoref{th:t-to-zero}) and an energy derivative formula (\autoref{pr:EKderiv}) that holds for all $t$. Those results therefore constitute the core of the paper. 

\subsection*{Physical motivations, and an open problem of P\'{o}lya and Szeg\H{o}}
Both endpoint limits in the theorem are physically motivated. When $t \to \infty$, the strip boundary recedes to infinity and ceases to affect the kernel, hence reducing $G_t$ to just the original Riesz kernel. When $t \to 0$, on the other hand, the insulated boundaries bracket both sides of the set $K$, suggesting that the kernel should be independent of the vertical direction and so the effective dimension and corresponding Riesz exponent should be reduced by $1$. This key intuition is made rigorous  in \autoref{pr:kernel}(d) and \autoref{pr:twosided}. 

Notably, for $(n,q)=(2,1)$ the theorem provides a physically natural one-parameter family of strip energies that connects the planar logarithmic energy of a set $K \subset \R^2$ to the Newtonian electrostatic energy of $K$ regarded as a subset of $\R^3$. For these energies, P\'{o}lya and Szeg\H{o} conjectured in 1945 that when passing from the logarithmic to Newtonian situation, the disk retains more capacity than any other planar set:   
\begin{equation} \label{PS2}
\capone(K) \leq \frac{2}{\pi} \capzero(K) 
\end{equation}
whenever $K \subset \R^2$ is compact, and equality holds when $K$ is a disk \cite[Conjecture (1.3)]{PS45}. Here $\capone(K)=1/V_1(K)$ is the Newtonian capacity of $K$ and $\capzero(K)=\exp(-V_{log}(K))$ is its logarithmic capacity. 

This conjecture remains open to the best of our knowledge. The family of strip energies in the current paper provides a promising framework for trying to resolve it, through the following stronger conjecture. 
\begin{conjecture} \label{co:framework}
Fix $1 \leq q < n$. Suppose $K \subset \Rn$ is compact and $B \subset \Rn$ is a closed ball. If $\lim_{t \to \infty} E_K(t) = \lim_{t \to \infty} E_B(t)$ then $E_K(t) \leq E_B(t)$ for all $t>0$. 
\end{conjecture}
Recalling from \autoref{th:main} the limiting values of the energy as $t \to \infty$ and $t \to 0$, we see \autoref{co:framework} would imply that if $V_q(K)=V_q(B)$ then $V_{q-1}(K) \leq V_{q-1}(B)$ when $q>1$ and $V_{log}(K) \leq V_{log}(B)$ when $q=1$. That conclusion is equivalent, when $n=2$ and $q=1$, to the P\'{o}lya--Szeg\H{o} conjecture \eqref{PS2}. 

Our recent work \cite{CL25b} investigated the P\'{o}lya--Szeg\H{o} capacity conjecture not just for Riesz parameters $q-1$ and $q$ as in \autoref{co:framework} but for the full range of exponents $p$ and $q$, including negative values. Some of the resulting conjectures can be proved, and the limiting cases with $p \to -\infty$ or $q \to n$ turn out (respectively) to be known theorems of Szeg\H{o} \cite{S31} for fixed diameter capacity maximization and of Watanabe \cite{W83} for fixed-volume capacity minimization. Another limiting case is the classical isodiametric theorem for maximizing volume among sets of fixed diameter. Thus evidence is building in support of the general conjecture, although the original problem \eqref{PS2} remains open. 

Incidentally, an unrelated capacity conjecture by P\'{o}lya and Szeg\H{o} asserts that among convex sets in $\R^3$ with given surface area, the Newtonian capacity ($q=1$) is minimal in the degenerate case of a double-sided disk. This problem too remains open, although excellent partial results have been achieved in recent years; for example, see \cite{BFL12,FGP11,X17}. 

\subsection*{\bf Plan of the paper} \autoref{sec:definitions} constructs the kernels and energies. In \autoref{sec:energyresults} we state the results mentioned above that are used to establish \autoref{th:main}. \autoref{sec:kernelproperties} develops basic properties of the strip kernel $G_t$, while \autoref{sec:twosided} gives upper and lower estimates on the kernel, \autoref{sec:improved} provides an asymptotic kernel estimate as $t \to \infty$, and then \autoref{sec:stripenergy} utilizes these tools to prove the results stated in \autoref{sec:energyresults}. To complete the paper, \autoref{sec:computableexample} provides a closed form expression for the kernel $G_t$ in $n=3$ dimensions when $q=2$. 

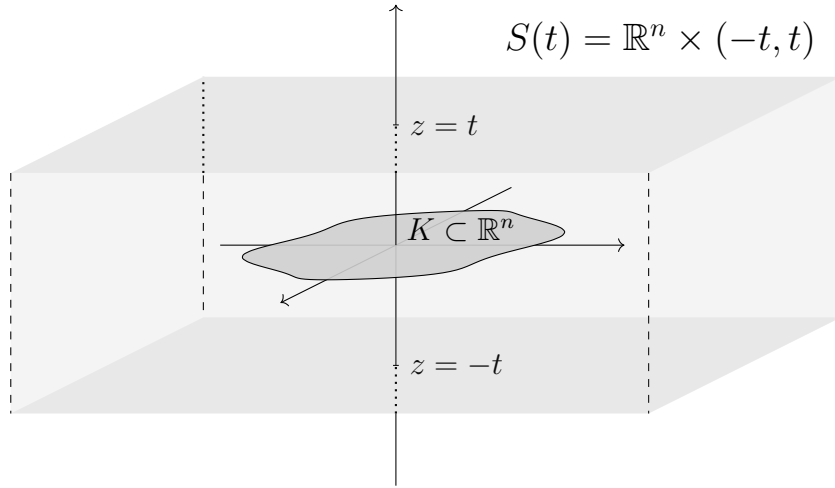
\begin{figure}
\begin{center}
\begin{tikzpicture}[x={(2cm,0cm)},y={(0.8cm,0.4cm)},z={(0cm,1cm)},scale=1.6]

\fill[gray!9] (1.45,1,1) -- (1.45,1,-1) -- (1.45,-1,-1) -- (-1.2,-1,-1) -- (-1.2,-1,1) -- (-1.2,1,1) -- cycle;
\draw[-,dashed] (1.45,1,1) -- (1.45,1,-1);
\draw[-,dashed] (1.45,-1,1) -- (1.45,-1,-1);
\draw[-,dashed] (-1.2,1,0.18) -- (-1.2,1,-1);
\draw[-,dotted,thick] (-1.2,1,0.2) -- (-1.2,1,1);
\draw[-,dashed] (-1.2,-1,1) -- (-1.2,-1,-1);

\draw[->] (-0.73,0,0) -- (0.95,0,0);
\draw[->] (0,1.2,0) -- (0,-1.2,0) ;
\draw[-] (0,0,0) -- (0,0,-1);
\draw[dotted, thick] (0,0,-1.4) -- (0,0,-1);
\draw[-] (0,0,-1.4) -- (0,0,-2);
\draw[-] (-0.012,0,1) -- (0.012,0,1) node[right] {$z=t$};
\draw[-] (-0.012,0,-1) -- (0.012,0,-1) node[right] {$z=-t$};
 \filldraw[fill opacity=0.9,gray!37,draw=black] plot[smooth,samples=50,domain=0:356,variable=\t] ({0.55*(cos(\t)+0.1*sin(5*\t)+0.05},{0.72*(sin(\t)+0.02*sin(2*\t))}) -- cycle;

\draw[-] (0,0,0) -- (0,0,0.6);
\draw[dotted, thick] (0,0,0.6) -- (0,0,1);
\draw[->] (0,0,1) -- (0,0,2);

\draw (1.1,0,1.7) node {\large$S(t)=\Rn\times (-t,t)$};
\draw (0.27,0,0.13) node {$K\subset\Rn$}; 

\fill[fill opacity=0.1,gray] (1.45,1,1) -- (1.45,-1,1) -- (-1.2,-1,1) -- (-1.2,1,1) -- cycle;
\fill[fill opacity=0.1,gray] (1.45,1,-1) -- (1.45,-1,-1) -- (-1.2,-1,-1) -- (-1.2,1,-1) -- cycle;
\end{tikzpicture}
\end{center}
\caption{
\label{fig:strip} 
The compact set $K \subset \Rn$ in the strip $S(t)=\Rn\times(-t,t)$. Points in the strip are written $\xh = (x,z)$. 
}
\end{figure}

\section{\bf Kernel and energy definitions}
\label{sec:definitions}

\subsection*{Riesz energy and capacity in all of space}
References for the following definitions and facts can be found in \cite{CL25}. 

Consider a compact nonempty set $K$ in a Euclidean space. (The dimension of the space does not enter into the following definitions and so need not be specified.) When $q > 0$, the Riesz $q$-energy of $K$ is
\[
V_q(K) = \min_\mu \int_K \! \int_K |x-y|^{-q} \, d\mu(x) d\mu(y) , 
\]
where the minimum is taken over all probability measures on $K$, that is, positive unit Borel measures. The minimum is attained by an ``equilibrium" measure. The energy is positive or $+\infty$, and if the energy is finite then the equilibrium measure is unique, in which case we call it the $q$-equilibrium measure of $K$. For the empty set we define $V_q(\emptyset)$ to equal $+\infty$. 

When $q=0$ we consider instead the logarithmic energy 
\[
V_{log}(K) = \min_\mu \int_K \! \int_K \log \frac{1}{|x-y|} \, d\mu(x) d\mu(y) ,
\]
with the minimum taken over probability measures on $K$. The minimum is attained by an equilibrium measure. The energy is greater than $-\infty$ since $|x-y|$ is bounded on $K \times K$. If the energy is less than $+\infty$ then the logarithmic equilibrium measure is unique and is called the $0$-equilibrium measure. For the empty set, define $V_{log}(\emptyset)=+\infty$. 

The Riesz capacity of $K$ is 
\begin{equation*} 
\capq(K) = 
\begin{cases}
V_q(K)^{-1/q} , & q > 0 , \\
\exp(-V_{log}(K)) , & q = 0 .
\end{cases}
\end{equation*}
The $0$-capacity is also called logarithmic capacity. Notice the Riesz capacity is positive if and only if the energy is finite. The definition also ensures that capacity is monotonic with respect to set inclusion ($K_1 \subset K_2$ implies $\capq(K_1) \leq \capq(K_2)$) and that capacity scales linearly, that is, $\capq(sK)=s \capq(K)$ whenever $s>0$. 

In this paper we work in $\Rnp$ and for that reason consider only $q<n+1$, since the capacity is known to vanish (energy is infinite) when $q$ is greater than or equal to the dimension of the ambient space. 

\begin{remark*} Properties of capacity as a function of $q$ are investigated in our recent paper \cite{CL25}. For example, $q \mapsto \capq(K)$ is strictly decreasing and left-continuous, but it can jump from the right. And in the endpoint case where $q$ approaches the dimension of the Euclidean space, Riesz capacity determines the volume of $K$ according to $\text{Vol}(K) = |\Sph^{n-1}| \lim_{q \nearrow n} \frac{\capq(K)^q}{n-q} $. Other recent developments for Riesz capacities can be found in a paper by Liu and Xiao \cite{LX25}. Known formulas for the Riesz capacity of a ball are collected in \cite[Appendix A]{CL25}. 
\end{remark*}

\subsection*{Riesz kernel on infinite strip with Neumann boundary conditions}
Fix $n \geq 1$ and let $t>0$. For each $j \in \Z$, construct a transformation 
\[
\rho_j(y,w)=(y,2tj+(-1)^j w)  
\] 
for $y \in \Rn, w \in \R$, so that $\rho_j$ is a vertical translation when $j$ is even and is a reflection and translation when $j$ is odd, as illustrated in \autoref{fig:yhat}. Although the mapping $\rho_j$ depends on $t$, we suppress that dependence for notational simplicity. Clearly $\rho_0$ is the identity map. 
\begin{definition*}[Strip kernel by method of reflections]
Let $t>0$ and $q>1$. Define a kernel on the closure of the strip 
\[
S(t)=\Rn \times (-t,t) 
\]
by letting 
\begin{equation} \label{eq:rieszkernelq}
G_t(\xh,\yh) = \sum_{j \in \Z} \frac{1}{| \xh-\rho_j(\yh) |^q} 
\end{equation}
when 
\[
\xh=(x,z), \ \yh=(y,w) \in \overline{S(t)} , 
\]
that is, when $x,y \in \Rn$ and $z,w\in[-t,t]$. Note that the series for $G_t$ is greater than zero everywhere, and converges locally uniformly wherever $\xh \neq \yh$, since $q >1$. Obviously the kernel $G_t$ depends on $q$, but we do not track that dependence since $q$ is fixed. We will see in \autoref{sec:kernelproperties} that the kernel satisfies Neumann conditions on the boundary planes of the strip.

In order to define the kernel when $q=1$, we renormalize the formula by subtracting a summand that cancels the divergent leading order term. Precisely, when $q=1$ we define
\begin{equation}\label{eq:rieszkernelone}
G_t(\xh,\yh) 
= \frac{1}{|\xh-\yh|}+ \sum_{j\neq 0}\left( \frac{1}{|\xh-\rho_j(\yh)|} -\frac{1}{|2tj|} \right) + \frac{1}{t}(\gamma-\log (4t))
\end{equation}
for $\xh, \yh \in \overline{S(t)}$, where $\gamma \simeq 0.577$ is the Euler--Mascheroni constant. The series converges locally uniformly where $\xh \neq \yh$, since $1/|\xh-\rho_j(\yh)|$ behaves like $1/|2tj| + O(1/j^2)$ for large $|j|$. Thus in particular, although we do not know that the kernel is positive when $q=1$, it is certainly bounded below provided $\xh$ and $\yh$ lie in some bounded subset of the closed strip. 

When $z=w=0$, so that $\xh=(x,0)$ and $\yh=(y,0)$, we sometimes abuse notation and write the kernel at those points simply as $G_t(x,y)$. 
\qed
\end{definition*}

The additive term $(\gamma-\log (4t))/t$ in definition \eqref{eq:rieszkernelone} when $q=1$ ensures there is no additive constant term in the asymptotic formula for the kernel when $|x-y|$ is large, in \autoref{pr:twosided} below. 

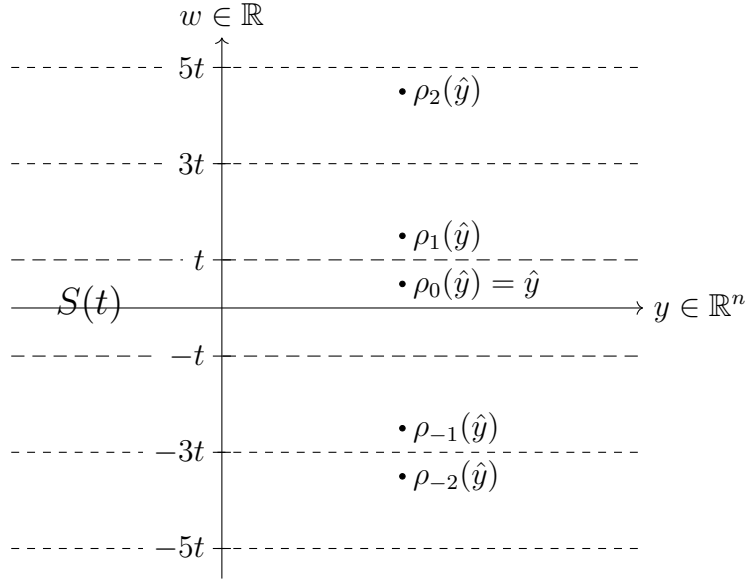
\begin{figure}
\begin{center}
\begin{tikzpicture}[scale=0.8]
\def\t{0.8};
\def\xmin{-3.5};
\def\xmax{7};
\def\zmin{-5*\t-0.5};
\def\zmax{5*\t+0.5};
\def\y{3};
\def\w{\t/2};
\draw[dash pattern=on 5pt off 3pt] (\xmin,\t) -- (\xmax,\t);
\draw[dash pattern=on 5pt off 3pt] (\xmin,-\t) -- (\xmax,-\t);\
\draw[dashed] (\xmin,3*\t) -- (\xmax,3*\t);
\draw[dashed] (\xmin,5*\t) -- (\xmax,5*\t);
\draw[dashed] (\xmin,-3*\t) -- (\xmax,-3*\t);
\draw[dashed] (\xmin,-5*\t) -- (\xmax,-5*\t);
\draw[black,fill=black] (\y,\w) circle (.25ex) node[right]{$\rho_{0}(\yh)=\yh$};
\draw[black,fill=black] (\y,{2*\t*1+\w*(-1)^1}) circle (.25ex) node[right]{$\rho_{1}(\yh)$};
\draw[black,fill=black] (\y,{2*\t*2+\w*(-1)^2}) circle (.25ex) node[right]{$\rho_{2}(\yh)$};
\draw[black,fill=black] (\y,{2*\t*(-1)+\w*(-1)^(-1)}) circle (.25ex) node[right]{$\rho_{-1}(\yh)$};
\draw[black,fill=black] (\y,{2*\t*(-2)+\w*(-1)^(-2)}) circle (.25ex) node[right]{$\rho_{-2}(\yh)$};
\draw (-0.1,\t) node[left,fill=white] {$t$};
\draw (-0.1,-\t) node[left,fill=white] {$-t$};
\draw (-0.1,3*\t) node[left,fill=white] {$3t$};
\draw (-0.1,5*\t) node[left,fill=white] {$5t$};
\draw (-0.1,-3*\t) node[left,fill=white] {$-3t$};
\draw (-0.1,-5*\t) node[left,fill=white] {$-5t$};
\draw[-] (-0.15,\t) -- (0.15,\t);
\draw[-] (-0.15,-\t) -- (0.15,-\t);
\draw[-] (-0.15,3*\t) -- (0.15,3*\t);
\draw[-] (-0.15,-3*\t) -- (0.15,-3*\t);
\draw[-] (-0.15,5*\t) -- (0.15,5*\t);
\draw[-] (-0.15,-5*\t) -- (0.15,-5*\t);
\draw (-2.2,0.05) node {\large$S(t)$};
\draw[->] (\xmin,0) -- (\xmax,0) node[right] {$y \in \Rn$};
\draw[->] (0,\zmin) -- (0,\zmax) node[above] {$w \in \R$};
\end{tikzpicture}
\end{center}
\caption{\label{fig:yhat} The heights of the points $\rho_j(\yh)$ are determined by repeated reflection across the strip boundaries at height $\pm t$. For example, $\rho_0(\yh)$ reflects to $\rho_{\pm 1}(\yh)$, which then reflect to $\rho_{\mp 2}(\yh)$, and so on. This method of reflections ensures that the kernel $G_t$ satisfies a Neumann condition at the strip boundaries.}
\end{figure}

\subsection*{Riesz energy on infinite strip under Neumann boundary conditions}
Fix $n \geq 1$ and $q \geq 1$ and consider a nonempty compact set $K \subset \Rnp$. Supposing $t>0$ is large enough that $K \subset S(t)$, we may define the strip energy of $K$ to be 
\[
E_K(t) = \min_\mu \int_K \! \int_K G_t(\xh,\yh) \, d\mu(\xh) d\mu(\yh) , 
\]
where the minimum is taken over all probability measures on $K$. See \cite[Lemma 4.1.3]{BHS19} for the existence of an equilibrium measure that achieves the minimum.  If $K$ is empty, define the energy to be $+\infty$.

If $K \subset \Rn \simeq \Rn \times \{0\}$ and the energy is finite then the equilibrium measure is unique and we denote it $\mu_t$. The uniqueness is proved below in \autoref{le:uniqueness}.

Remember that $G_t$ and the energy $E_K(t)$ depend upon the choice of exponent $q$ in the kernel. Because $q$ is fixed, we omit this dependence from the notation.

In the limiting case $t \to \infty$ the strip fills all of $\Rnp$ and so when $q \geq 1$ it makes sense to define $G_\infty$ to be the $q$-Riesz kernel, 
\[
G_\infty(\xh,\yh)=\frac{1}{|\xh-\yh|^q} .
\]
Hence by definition
\[
E_K(\infty)=V_q(K)
\]
is simply the $q$-Riesz energy of $K$. 

We conclude the section by addressing some uniqueness and finiteness issues.

\begin{lemma}[Finite energy for a set in the hyperplane $\Rn$ implies unique equilibrium measure] \label{le:uniqueness}
Suppose $q\geq 1$ and $t>0$. If $K$ is a compact subset of $\Rn$ and $E_K(t)$ is finite then the equilibrium measure $\mu_t$ for the energy $E_K(t)$ is unique. 
\end{lemma}
\begin{proof}
First, consider the case $q>1$, in which case for $x,y \in \Rn$ the kernel simplifies to the form
\[
G_t(x,y) = \sum_{j\in \Z} \frac{1}{(|x-y|^2+(2tj)^2)^{q/2}} = f(|x-y|^2),
\]
where
\[
f(a) =  \sum_{j\in \Z} \frac{1}{(a+(2tj)^2)^{q/2}}, \qquad a\geq 0.
\]
 By \cite[Theorems 4.2.7 and 4.4.4]{BHS19} it is enough to check that $f$  is strictly completely monotone, that is $(-1)^kf^{(k)}(a)> 0$ for all $a>0,$ for each $k\geq 0$ (see \cite[Definition 2.2.4]{BHS19}). These inequalities are readily verified by direct computation of the derivatives.

In the case $q=1$, for $x,y \in\Rn$ the kernel simplifies to
\[
G_t(x,y)=\frac{1}{|x-y|}+\sum_{j\neq 0}\left(\frac{1}{(|x-y|^2+(2tj)^2)^{1/2}}-\frac{1}{|2tj|} \right) +\frac{1}{t}(\gamma-\log(4t)),
\]
which equals $f(|x-y|^2)$ where
\[
f(a) = \frac{1}{\sqrt{a}}+\sum_{j\neq 0} \left( \frac{1}{(a+(2tj)^2)^{1/2}}-\frac{1}{|2tj|}\right)+\frac{1}{t}(\gamma-\log(4t)), \qquad a \geq 0.
\]
By \cite[Theorem 4.2.7, Lemma 4.4.6, and Equation 2.2.5]{BHS19} it is sufficient to prove that $-f'$ is (strictly) completely monotone, which is straightforward to check by computing the derivatives of 
\[
-f'(a)= \sum_{j\in \Z} \frac{1}{(a+(2tj)^2)^{3/2}}.
\]

\end{proof}
\begin{lemma}[Finite energy for one strip implies finiteness and continuity for all] \label{le:finiteness}
Suppose $q \geq 1$ and $0<t_1<t_2 \leq \infty$. If $K$ is a compact subset of $S(t_1)$ then  
\[
E_K(t_1) < \infty \quad \Longleftrightarrow \quad E_K(t_2) < \infty .
\]
Furthermore, if $E_K(t_1) < \infty$ then $E_K(t)$ is continuous for all $t \in [t_1,\infty]$. 
\end{lemma}
\begin{proof}
The kernel decomposes as  
\[
G_t(\xh,\yh)=G_\infty(\xh,\yh)+H_t(\xh,\yh) 
\]
where  
\begin{equation} \label{eq:Hdef}
H_t(\xh,\yh) = 
\begin{cases}
\sum_{j\neq 0} \dfrac{1}{ |\xh-\rho_j(\yh)|^{q}} & \text{when $q>1$,} \\[1em]
\sum_{j\neq 0} \! \left( \dfrac{1}{|\xh-\rho_j(\yh)|} -\dfrac{1}{|2tj|} \right) + \dfrac{\gamma-\log (4t)}{t} & \text{when $q=1$.}
\end{cases}
\end{equation}
Recall that $\rho_j(\yh)=(y,2tj+(-1)^j w)$ depends on $t$. When $K \subset S(t)$, the energy $E_K(t)$ differs by at most a bounded amount from $E_K(\infty)$, because $H_t(\xh,\yh)$ is bounded for $\xh,\yh \in K$. Thus $E_K(t)$ is finite if and only if $E_K(\infty)$ is finite, and the first part of the lemma follows. 

Assuming the energy is finite, continuity of $E_K(t)$ for $t \in [t_1,\infty)$ is clear since $H_t(\xh,\yh)$ is jointly continuous as a function of $\xh,\yh \in K$ and $t>0$, so that small changes in $t$ induce only small changes in the kernel and hence in the energy. To justify continuity as $t \to \infty$, when $q>1$ we see $H_t(\xh,\yh)$ is bounded by $O(1/t^q) \sum_{j \neq 0} 1/|j|^q$ uniformly for $\xh,\yh \in K$. For $q=1$, the series in $H_t(\xh,\yh)$ is bounded by $O(1/t^2) \sum_{j \neq 0} 1/|j|^2$ uniformly for $\xh,\yh \in K$. The term $(\gamma-\log(4t))/t$ in $H_t(\xh,\yh)$ also tends to zero as $t \to \infty$. Thus for each $q$, it follows that $G_t(\xh,\yh)$ converges locally uniformly to $G_\infty(\xh,\yh)$. Hence $E_K(t) \to E_K(\infty)$ as $t \to \infty$. 
\end{proof}
\begin{remark*} 
We restrict attention from now on to energies with $q<n+1$, for sets in $\Rnp$, because if $q \geq n+1$ then the strip energy always equals $+\infty$, which is not interesting: indeed, $q\geq n+1>1$ implies $V_q(K)=+\infty$ by \cite[Theorems 4.3.1 and 4.3.3]{BHS19} (see also \cite[Theorem 1.1.(d)]{CL25}), and since $G_t \geq G_\infty$, we conclude $E_K(t)=+\infty$.

Similarly, for sets in $\Rn$ we will restrict attention to energies with $q<n$. 
\end{remark*}

\section{\bf Results on the Neumann strip energy}
\label{sec:energyresults}

This section states the three core results that together imply \autoref{th:main}. 
\begin{itemize}
	\item \autoref{th:energybigt}: asymptotic to third order for the Neumann strip energy as $t\to \infty$, in terms of the $q$-equilibrium measure. 
	\item \autoref{th:t-to-zero}: leading order asymptotic for the strip energy of ``nice'' sets as $t\to 0$, yielding the $(q-1)$-energy. 
	\item \autoref{pr:EKderiv}: derivative of the strip energy with respect to $t$. 
\end{itemize}
After we construct the requisite tools in the next few sections, the three results will be proved in \autoref{sec:stripenergy}. 

\smallskip
The first theorem says that the strip energy $E_K(t)$ converges as $t \to \infty$ to the whole-space Riesz energy $E_K(\infty)=V_q(K)$. That limiting value is known already from \autoref{le:finiteness}, but the theorem sharpens our understanding by identifying also the next two terms in the asymptotic formula. Those terms seem likely to be useful in resolving P\'{o}lya and Szeg\H{o}'s capacity conjecture \eqref{PS2}. 

Remember our standard notations $\xh=(x,z)$ and $\yh=(y,w)$.
\begin{theorem}[Energy asymptotic as $t \to \infty$, for $K \subset \Rnp$] \label{th:energybigt}
Fix $1\leq q <n+1$. Let $K$ be a compact set in $\Rnp$ that has finite Riesz $q$-energy. Then
\[
E_K(t)=E_K(\infty)+\frac{A_q(t)}{(2t)^q}-\frac{1}{(2t)^{q+2}}\left(B_q \, M_q(K)-C_q \left( \int_K  z \, d\mu_{\infty}(\xh)\right)^{\!2}\right)+o\! \left(\frac{1}{t^{q+2}}\right)
\]
as $t \to \infty$, where the coefficients in the expansion are 
\[
A_q(t)=
\begin{cases}
2\zeta(q) = \sum_{j\neq 0} \frac{1}{|j|^{q}} & \text{for $q>1$,} \\
2(\gamma-\log(4t)) & \text{for $q=1$,}
\end{cases}
\]
\[
 B_q=\frac{q}{2}\sum_{j\neq 0}\frac{1}{|j|^{q+2}}=q \, \zeta(q+2), \quad C_q=2q(q+1)\sum_{j\,\text{odd} }\frac{1}{|j|^{q+2}},
\]
\[
M_q(K)=\int_K\int_K(|x-y|^2-(q+1)(z-w)^2)\,d\mu_{\infty}(\xh)\,d\mu_{\infty}(\yh) ,
\]
and $\mu_{\infty}$ is the Riesz $q$-equilibrium measure of $K$.
\end{theorem}
Note that $A_q(t)$ depends on $t$ only in the case $q=1$. It is a constant when $q>1$. 

When $K$ is contained in the lower dimensional space $\Rn$, the $z$ and $w$ coordinates equal $0$ on $K$ and so one arrives at the following simpler asymptotic expression involving the second moment of equilibrium measure.

\begin{corollary}[Energy asymptotic as $t \to \infty$, for $K \subset \Rn$] \label{co:energybigtKinRn}
Fix $1\leq q <n$. If $K$ is a compact set in $\Rn$ that has finite Riesz $q$-energy then
\[
E_K(t)=E_K(\infty)+\frac{A_q(t)}{(2t)^q}-\frac{2B_q}{(2t)^{q+2}}\int_K|x-x_{K,\infty}|^2\,d\mu_{\infty}+o\! \left(\frac{1}{t^{q+2}}\right)
\]
as $t \to \infty$, where $x_{K,\infty}=\int_K x\, d\mu_{\infty}$ is the centroid of the $q$-equilibrium measure $\mu_{\infty}$.
 \end{corollary}
Recently we investigated this second moment $\int_K|x-x_{K,\infty}|^2\,d\mu_{\infty}$ of equilibrium measure in the case $q=n-1$, showing for this electrostatic $q$-value in \cite[Theorem 1.2]{CL24} that among sets in $\Rn$ of given energy $E_K(\infty)$, the second moment is strictly minimal for $K$ an $n$-dimensional ball $B$. Thus if $K$ is not a ball then $E_K(t) < E_B(t)$ for all large $t$ by \autoref{co:energybigtKinRn}, which is consistent with our \autoref{co:framework}.  

The next theorem shows that for sets contained in the central hyperplane $\Rn$, the strip energy $E_K(t)$ is bounded below by the Riesz $(q-1)$-energy, up to a constant factor that we denote by 
\begin{equation} \label{cdef}
c_{q-1} = \frac{1}{2} \int_{-\infty}^\infty \frac{1}{\left( 1 + s^2 \right)^{q/2}} \, ds = \frac{\Gamma(\frac{1}{2}) \Gamma(\frac{q-1}{2})}{2\Gamma(\frac{q}{2})} , \qquad q>1 .
\end{equation}
For example, $c_1=\pi/2$ and $c_2=1$. 

In the next result, we say that a compact set $K \subset \Rn$ is \textit{interior $q$-capacitable in $\Rn$} if its capacity is fully captured by the interior, meaning that 
\[
\capq(K)=\sup\{ \capq(Q) \, : \, Q \subset K^{int} \text{ is compact} \}.
\]
For example, if $\lambda K$ is contained in the interior of $K$ for all $0<\lambda<1$ then $K$ is interior $q$-capacitable by choosing $Q=\lambda K$ and letting $\lambda$ tend to $1$. In particular, starlike and convex bodies in $\Rn$ are interior capacitable. 
\begin{theorem}[Limiting energy as $t \to 0$] \label{th:t-to-zero}
Fix $1 \leq q < n$ and let $K \subset \Rn$ be compact.

(a) For all $t>0$, 
\[
t E_K(t) \geq 
\begin{cases}
c_{q-1} V_{q-1}(K) & \text{when $q > 1$,} \\
V_{log}(K) & \text{when $q=1$.}
\end{cases}
\]

(b) If $K$ is interior $(q-1)$-capacitable in $\Rn$ then 
\[
\limsup_{t \to 0} t E_{K}(t) \leq 
\begin{cases}
c_{q-1} V_{q-1}(K) & \text{when $q >1$,} \\
V_{log}(K) & \text{when $q=1$.}
\end{cases}
\]

(c) If $K$ is interior $(q-1)$-capacitable in $\Rn$ then equality is achieved as $t\to 0$:
\[
\lim_{t \to 0} t E_{K}(t) = 
\begin{cases}
c_{q-1} V_{q-1}(K) & \text{when $q > 1$,} \\
V_{log}(K) & \text{when $q=1$.}
\end{cases}
\]
\end{theorem}
Part (c) follows immediately from the lower and upper bounds in parts (a) and (b). 

A hypothesis such as interior capacitability is needed for part (b), because the upper bound there is false in general. For a concrete example, suppose $K$ is a line segment in the plane ($n=2$), and recall that a line segment has zero Newtonian capacity ($q=1$) and positive logarithmic capacity ($q=0$). The segment has empty interior in $\R^2$ and so is not interior $0$-capacitable. Notice $E_K(t)=+\infty$ by \autoref{le:finiteness} because $E_K(\infty)=V_1(K)=\infty$, while $V_{log}(K)<\infty$ and so the inequality in (b) fails for the line segment. Or for a whole family of examples, suppose $1<q<n$ and $K \subset \Rn$ has Hausdorff dimension strictly between $q-1$ and $q$; then $V_q(K)=+\infty$ (by \cite[Theorem 4.3.1]{BHS19}) and so $E_K(t)=+\infty$ for all $t$, which cannot be bounded above by $c_{q-1} V_{q-1}(K)$ since that number is finite (by \cite[Theorem 4.3.3]{BHS19}).

\subsection*{Energy derivative wrt strip thickness $t$}
\label{sec:energyderiv}

Now that we understand the energy limits as $t \to \infty$ and $t \to 0$, we analyze the derivative $E_K^\prime(t)$ at a finite $t$. 
\begin{theorem}[Energy derivative] \label{pr:EKderiv}
Suppose $1 \leq q < n$ and $K \subset \Rn$ is compact. If $V_q(K)$ is finite, then $E_K(t)$ is finite and continuously differentiable for $t > 0$, with 
\[
E_K^{\, \prime}(t) = \int_K \! \int_K \frac{\partial G_t}{\partial t}(x,y) \, d\mu_t(x) d\mu_t(y)  
\]
where $\mu_t$ is the equilibrium measure for $E_K(t)$.
\end{theorem}
The theorem says that one may differentiate through the energy integral $E_K(t) = \int_K \! \int_K G_t(\xh,\yh) \, d\mu_t d\mu_t$ and place the $t$-derivative onto the kernel while ignoring the $t$-dependence of the equilibrium measure $\mu_t$. This procedure is plausible since the first variation measure $\dot{\mu}_t$ (if it exists) would be a signed measure with $\dot{\mu}_t(K)=0$ (remember $\mu_t(K)=1$ for all $t$) while the equilibrium potential $\int_K G_t(\xh,\yh) \, d\mu_t(\xh)$ should be constant on the support of $\mu_t$ except perhaps for a set of capacity zero (see Landkof \cite[p.{\,}137]{L72}). That is not how we will prove the theorem, though, because we do not know whether the first variation $\dot{\mu}_t$ exists. 

Instead we view the theorem as ``differentiating through the minimum of the energy functional'', which is a well known principle in shape optimization theory dating back at least to Danskin's envelope theorem \cite[Theorem 1]{D66}. Danskin's theorem has given birth to many extensions and variants, e.g.\ \cite{C75,MS02,OT18} and \cite[Theorem 10.2.1]{DZ11}, but to our knowledge the first variant that was applicable to potential theoretic energy functionals was by Laugesen \cite[Theorem 1]{L93}. The proof  of \autoref{pr:EKderiv} in \autoref{sec:stripenergy} follows analogous lines. Such potential theoretic derivative formulas have subsequently been generalized by Pouliasis \cite[Theorem 1.1]{P21}.

\section{\bf Strip kernel: basic properties}\label{sec:kernelproperties}

This section is devoted to proving:
\begin{proposition}[Basic properties of strip kernel] \label{pr:kernel} Consider the kernel $G_t$ on the strip $S(t) \subset \Rnp$, where $t>0$ and $n, q \geq 1$. 

(a) The kernel is symmetric: 
\[
G_t(\xh,\yh) = G_t(\yh,\xh) , \qquad \xh, \yh \in S(t) .
\]

(b) Its normal derivative vanishes on the top and bottom of the strip, where $z=\pm t$: 
\[
\left. \frac{\partial G_t}{\partial z} \right|_{z=\pm t} \!\! = 0 , \qquad x , y \in \Rn, \quad w \in (-t,t) .
\]

(c) As $t \to \infty$, the kernel converges pointwise to the whole-space Riesz kernel:
\[
\lim_{t \to \infty} G_t(\xh,\yh) = G_{\infty}(\xh,\yh) = \frac{1}{|\xh-\yh|^q} .
\]

(d) If $\xh=(x,0)$ and $\yh=(y,0)$ lie in the central hyperplane $\Rn$ then $t G_t(x,y)$ is bounded below by the Riesz kernel with parameter $q-1$: 
\begin{equation} \label{eq:kernelcomparison}
t G_t(x,y) > 
\begin{cases}
c_{q-1} \dfrac{1}{|x-y|^{q-1}} , & q > 1 , \\[1em]
\log \dfrac{1}{|x-y|} , & q=1 ,
\end{cases} 
\end{equation}
for all $x,y \in \Rn, x \neq y$.
\end{proposition}

The limiting case of thick strips in parts (c) and the lower bound for points in the central hyperplane in part (d) are central to our goal of showing that the family of strip kernels interpolates between Riesz kernels with exponents $q-1$ and $q$. 
\begin{proof}[Proof of \autoref{pr:kernel}]
Part (a). The definition of the transformation $\rho_j$ implies 
\begin{align*}
\xh-\rho_j(\yh) 
& = (x-y,z-2tj-(-1)^j w) \\
& = -
\begin{cases}
(y-x,w-2t(-j)-(-1)^j z) & \text{if $j$ is even,} \\
(y-x,-w+2tj+(-1)^j z) & \text{if $j$ is odd.}
\end{cases}
\end{align*}
Symmetry of the kernel in part (a) now follows from the observation that  
\[
|\xh-\rho_j(\yh)| 
= 
\begin{cases}
|\yh-\rho_{-j}(\xh)| & \text{if $j$ is even,} \\
|\yh-\rho_j(\xh)| & \text{if $j$ is odd.} 
\end{cases}
\]

\smallskip
Part (b). The normal derivative at the upper boundary of the strip is 
\[
\left. \frac{\partial\ }{\partial z} G_t(\xh,\yh) \right|_{z=t} = q \sum_{j \in \Z} \frac{(2j-1)t+(-1)^j w}{| (x-y,(2j-1)t+(-1)^j w) |^{q+2}} .
\]
The terms in the series cancel in pairs, by considering $j=0,1$, then $j=-1,2$, and $j=-2,3$, and so on. An analogous calculation works at the lower boundary, where $z=-t$. 

\smallskip
Part (c). This pointwise convergence of the kernel as $t \to \infty$ was established already in the proof of \autoref{le:finiteness}. In fact, that proof shows the convergence is locally uniform.

\smallskip
Part (d) for $q > 1$. By substituting $z=w=0$ into the definition \eqref{eq:rieszkernelq} of $G_t$, we find for $x \neq y$ that 
\begin{align*}
t G_t(x,y) 
& = \sum_{j \in \Z} \frac{t}{\left( |x-y|^2 + (2tj)^2 \right)^{q/2}} \notag \\
& = \frac{1}{2} \sum_{j \in \Z} \frac{2t/|x-y|}{\left( 1 + (2tj/|x-y|)^2 \right)^{q/2}} \cdot \frac{1}{|x-y|^{q-1}} . 
\end{align*}
Thus to prove the kernel comparison \eqref{eq:kernelcomparison}, we want to show 
\begin{equation} \label{eq:desired}
\sum_{j \in \Z} \frac{\tau}{\left( 1 + (j\tau)^2 \right)^{q/2}} > \int_\R \frac{ds}{\left( 1 + s^2 \right)^{q/2}} 
\end{equation}
where 
\[
\tau= \frac{2t}{|x-y|} .
\]
This inequality \eqref{eq:desired} claims that a Riemann sum with step size $\tau$ exceeds its corresponding integral. 

Let 
\[
g(s) = \frac{1}{(1+s^2)^{1/2}}
\]
and 
\begin{equation} \label{eq:his}
h(s)=g(s\tau) \tau = \frac{\tau}{(1+(s\tau)^2)^{1/2}} ,
\end{equation}
so that $h(s)$ decays like $1/|s|$ at infinity. The desired inequality \eqref{eq:desired} is equivalent to 
\[
\sum_{j \in \Z} h(j)^q > \int_\R h(s)^q \, ds .
\]
The Poisson summation formula is applicable to $h^q$ because $h(s)^q \sim 1/|s|^q$ as $|s| \to \infty$ (note $q > 1$) and the Fourier transform of $h^q$ decays rapidly due to smoothness of $h^q$. After applying Poisson summation to the left side of the last inequality (using the Fourier transform with $2\pi$ in the exponential), the task becomes to show 
\[
\sum_{m \in \Z} \widehat{h^q}(m) > \widehat{h^q}(0) .
\]

The term with $m = 0$ on the left cancels the sole term on the right, and so it suffices to show $\widehat{h^q}(\zeta)$ is positive when $\zeta \neq 0$. We verify this for $q \geq 1$. Note that $g^q$ has Fourier transform
\[
\widehat{g^q}(\sigma) = \frac{2\pi^{q/2}|\sigma|^{(q-1)/2}}{\Gamma(q/2)}K_{(q-1)/2}(2\pi |\sigma|) 
\]
by \cite[eq.{\,}10.32.11]{DLMF}, where $K_{(q-1)/2}$ is the $\frac{q-1}{2}$-th modified Bessel function of the second kind. It is known that $K_{(q-1)/2}(z)>0$ whenever $z > 0$; see \cite[eq.{\,}10.32.9]{DLMF}.  Hence $\widehat{h^q}(\zeta)>0$ whenever $\zeta \neq 0$, as we wanted to show.

\emph{Aside:} Alternatively, one can show positivity of the Fourier transform of $g^q$ as follows. Clearly, $g(s)^q=(1+s^2)^{-q/2}$ is a completely monotone function of $s^2$. Bernstein’s characterization of completely monotone functions in terms of exponentials \cite[eq.\,2.2.5]{BHS19} yields that $g(s)^q=\int_0^{\infty} e^{-s^2a} d\eta(a)$, for some measure $\eta$. Positivity of the Fourier transform quickly follows, since the transform of a Gaussian is again Gaussian.  

\smallskip
Part (d) for $q=1$. Again take $z=w=0$ and $x,y \in \Rn, x \neq y$. We will show that 
\begin{equation} \label{q1kernelformula}
2t G_t(x,y) - 2 \log \frac{1}{|x-y|} = \sum_{j \in \Z} \left( h(j) - \int_{j-1/2}^{j+1/2} h(s) \, ds \right)
\end{equation}
and that the right side of \eqref{q1kernelformula} is positive:  
\begin{equation} \label{riemann1}
\sum_{j \in \Z} \left( h(j) - \int_{j-1/2}^{j+1/2} h(s) \, ds \right) > 0 .
\end{equation}
Part (d) of the proposition then follows immediately. 

First we prove \eqref{q1kernelformula}. By rewriting the definition \eqref{eq:rieszkernelone} of the kernel in terms of the function $h$ and the quantity $\tau= 2t/|x-y|$, similar to our work above for $q>1$, one finds that the difference of the left side of \eqref{q1kernelformula} minus the right side equals
\begin{equation} \label{eq:equalszero}
\int_{-1/2}^{1/2} h(s) \, ds+ \sum_{j \neq 0} \left( \int_{j-1/2}^{j+1/2} h(s) \, ds - \frac{1}{|j|} \right) + 2(\gamma - \log 2\tau) .
\end{equation}
We will show this expression equals $0$. Begin by dividing it by $2$ and using evenness of $h$. The resulting expression is then equivalent to 
\[
\sum_{j=1}^\infty \left( \int_{j-1/2}^{j+1/2} \frac{1}{s} \, ds - \frac{1}{j} \right) + \gamma - \log 2 ,
\]
as one can check by taking the difference of the two expressions and evaluating the integrals explicitly, noting that $h(s)-1/s$ is integrable for large $s$ and has antiderivative $\log(\tau+\sqrt{s^{-2}+\tau^2})$. Since $\log 2 = \int_{1/2}^1 1/s \, ds$, the last line becomes 
\[
\lim_{k \to \infty} \left( \int_1^{k+1/2} \frac{1}{s} \, ds - \sum_{j=1}^k \frac{1}{j} \right) + \gamma ,
\]
which equals zero by definition of the Euler--Mascheroni constant $\gamma$. Thus expression \eqref{eq:equalszero} equals zero and so \eqref{q1kernelformula} is proved. 

For the positivity estimate \eqref{riemann1}, we note the left side equals 
\begin{align}
\sum_{j=-\infty}^\infty \int_{j-1/2}^{j+1/2} \int_s^j h^\prime(\xi) \, d\xi ds \notag 
& = \sum_{j=-\infty}^\infty \int_{j-1/2}^{j+1/2} f(\xi-j) h^\prime(\xi) \, d\xi \notag \\
& = \int_{-1/2}^{1/2} f(\xi) \sum_{j=-\infty}^\infty h^\prime(\xi+j) \, d\xi \label{riemann3}
\end{align}
where $f$ is the sawtooth function 
\[
f(\xi) = 
\begin{cases}
\xi + 1/2 , & -1/2 < \xi <0 , \\
\xi - 1/2, & 0 < \xi < 1/2 .
\end{cases}
\]
Parseval's identity for $1$-periodic functions reduces expression \eqref{riemann3} to $\sum_\ell \overline{\widehat{f}(\ell)} \widehat{h^\prime}(\ell)$, where we used that by periodization the $\ell$-th Fourier coefficient of the $1$-periodic function $\sum_j h^\prime(\xi+j)$ equals the Fourier transform $h^\prime(\ell)$. Thus to prove \eqref{riemann1}, we want to show the inequality 
\[
\sum_{\ell \neq 0} \overline{\widehat{f}(\ell)} \, (2\pi i \ell \, \widehat{h}(\ell)) > 0 . 
\]
Direct calculation of the Fourier coefficients of $f$ gives that $\overline{\widehat{f}(\ell)} 2\pi i \ell=1$ when $\ell \neq 0$, and so the desired inequality becomes $\sum_{\ell \neq 0} \widehat{h}(\ell) > 0$. This positivity indeed holds true because, as observed earlier, the transform $\widehat{h}$ is positive. That completes the proof of \eqref{riemann1} and hence of \autoref{pr:kernel}(d) when $q=1$. 
\end{proof}

\section{\bf Strip kernel: two-sided estimates} \label{sec:twosided}

The previous section developed a one-sided estimate on the strip kernel, that is, a lower bound in terms of the Riesz kernel with exponent $q-1$. Two-sided estimates on the strip kernel are provided by the next result. The price to be paid is that the error term now behaves like a Riesz kernel with the larger exponent $q$. Hence the estimate is useful only when $|x-y|$ is large. 
\begin{proposition}[Two-sided bounds in terms of horizontal $(q-1)$-Riesz kernel] \label{pr:twosided} 
Let $n, q \geq 1$. The kernel $G_t$ satisfies  
\begin{equation} \label{eq:Gttwosided}
G_t(\xh,\yh) = 
\begin{cases}
\dfrac{c_{q-1}}{t} \dfrac{1}{|x-y|^{q-1}} + O \! \left( \dfrac{1}{|x-y|^q} \right) , & q>1 , \\[1em]
\dfrac{1}{t} \log \dfrac{1}{|x-y|} + O \! \left( \dfrac{1}{|x-y|} \right) , & q=1 ,
\end{cases}
\end{equation}
and the horizontal and vertical components of its gradient satisfy 
\[
\nabla_{\! x} G_t(\xh,\yh) + \frac{q c_{q+1}}{t} \frac{x-y \quad }{|x-y|^{q+1}} = O\!\left( \frac{1}{|x-y|^{q+1}}\right), \qquad 
\frac{\partial\ }{\partial z} G_t(\xh,\yh)=O\!\left( \frac{1}{|x-y|^{q+1}} \right) ,
\]
for all $\xh=(x,z)$ and $\yh=(y,w)$ in $\overline{S(t)}$ with $x \neq y$ and all $t>0$. 
\end{proposition}
The constants implicit in the $O(\cdot)$ estimates depend on $q$, but are independent of $n,\xh,\yh,t$. In the horizontal gradient estimate the constant can be rewritten as   
\[
q c_{q+1} = 
\begin{cases}
(q-1) c_{q-1}, & q>1 , \\
1 , & q=1 ,
\end{cases}
\]
by definition \eqref{cdef} and the functional equation for the gamma function. 
\begin{proof}[Proof of \autoref{pr:twosided}]  

When $q>1$, the kernel definition \eqref{eq:rieszkernelq} implies that
\[
2t G_t(\xh,\yh) = \sum_{j \in \Z} g \left( j\tau + \delta_j \right)^q \tau \cdot \frac{1}{|x-y|^{q-1}}
\]
where $g(s) = 1/(1+s^2)^{1/2}$ and $\tau=2t/|x-y|$ and  
\[
\delta_j = \frac{(-1)^j w-z}{|x-y|} .
\] 
Notice $|\delta_j| \leq \tau$, since $w,z \in [-t,t]$

Thus to prove the kernel bound \eqref{eq:Gttwosided} when $q>1$ we want to show a Riemann sum approximation
\[
\sum_{j \in \Z} g \left( j\tau + \delta_j \right)^q \tau = 2c_{q-1} + O(\tau) 
\]
whenever $|\delta_j| \leq \tau$, with the error term satisfying $|O(\tau)| \leq (\text{const.}) \tau$ for some constant that depends only on $q$. Using the definition of $c_{q-1}$ from \eqref{cdef}, we see that 
\begin{align*}
\left| \sum_{j \in \Z} g \left( j\tau + \delta_j \right)^q \tau - 2c_{q-1} \right|
& = \left| \sum_{j \in \Z} \int_{j\tau}^{(j+1)\tau} \left( g( j\tau + \delta_j)^q - g(s)^q \right) ds \right| \\
& \leq \sum_{j \in \Z} \int_{j\tau}^{(j+1)\tau} \int_{(j-1)\tau}^{(j+1)\tau} \left| (g^q)^\prime(\xi) \right| \, d\xi ds \\
& \qquad  \text{ since } s \text{ and } j\tau + \delta_j \text{ lie between } (j\pm1)\tau\\
& = 2\tau \lVert (g^q)^\prime \rVert_{L^1} = O(\tau) ,
\end{align*}
as wanted.

Now suppose $q=1$. The kernel definition \eqref{eq:rieszkernelone} implies that 
\begin{align}
2t G_t(\xh,\yh) - 2 \log \frac{1}{|x-y|} 
& = g(\delta_0) \tau + \sum_{j \neq 0} \left( g \left( j\tau + \delta_j \right) \tau - \frac{1}{|j|} \right) + 2(\gamma - \log 2\tau) \notag \\
& = \sum_{j \in \Z} \int_{(j-1/2)\tau}^{(j+1/2)\tau} \left( g \left( j\tau + \delta_j \right) - g(s) \right) ds \label{eq:q1errorestimate}
\end{align}
by substituting for $2(\gamma - \log 2\tau)$ from expression \eqref{eq:equalszero} (which was shown to equal zero) and then changing variable in the integrals for the function $h(s)=g(s\tau)\tau$. That the quantity in \eqref{eq:q1errorestimate} is bounded by $2\tau \lVert g^\prime \rVert_{L^1}$ is proved as above for the case $q>1$, noting that $g^\prime$ is indeed integrable on $\R$.

For the horizontal gradient bound in the proposition, observe that for $q \geq 1$, differentiating the kernel definition \eqref{eq:rieszkernelq} gives
\[
\nabla_{\! x} G_{t,q}(\xh,\yh) = -q (x-y) G_{t,q+2}(\xh,\yh)
\] 
where the additional subscripts on the $G_t$'s indicate the $q$-value to be used in each kernel. On the right side, simply substitute for $G_{t,q+2}$ the estimate \eqref{eq:Gttwosided} that was proved already, with $q$ there replaced by $q+2$. 

For the vertical gradient bound, we find 
\begin{align*}
& \frac{\partial\ }{\partial z} G_t(\xh,\yh) \\
& = q \sum_{j \in \Z} \frac{2tj+(-1)^jw-z}{|\xh-\rho_j(\yh)|^{q+2}} \\
& = \frac{q}{|x-y|^{q+1}} \sum_{j \in \Z} (j\tau+\delta_j) g(j\tau+\delta_j)^{q+2} \\
& = \frac{q}{|x-y|^{q+1}} \left( \delta_0 \, g(\delta_0)^{q+2} + \sum_{j > 0} \left[ (j\tau+\delta_j) g(j\tau+\delta_j)^{q+2} - (j\tau-\delta_j) g(j\tau-\delta_j)^{q+2} \right] \right) 
\end{align*}
by combining the sums over positive and negative $j$-values and using evenness of $g$, noting that $-j\tau+\delta_{-j}=-(j\tau-\delta_j)$ because $\delta_{-j}=\delta_j$. Since $|\delta_j| \leq \tau$, it follows that 
\begin{align*}
\left| \frac{\partial\ }{\partial z} G_t(\xh,\yh) \right| 
& \leq \frac{q}{|x-y|^{q+1}} \left( \int_0^\tau |(sg(s)^{q+2})^\prime| \, ds + \sum_{j > 0} \int_{(j-1)\tau}^{(j+1)\tau} |(sg(s)^{q+2})^\prime| \, ds \right) \\
& = \frac{2q}{|x-y|^{q+1}} \int_0^\infty |(sg(s)^{q+2})^\prime| \, ds = O\!\left( \frac{1}{|x-y|^{q+1}} \right) .
\end{align*}
\end{proof}

Now consider a strip of fixed thickness. The decay of the kernel as $\xh$ approaches infinity can be deduced as follows, showing that the original decay exponent $q$ in $\Rnp$ gets reduced to $q-1$ in the strip. 
\begin{corollary}[Kernel decay at spatial infinity] \label{co:kerneldecay} Let $n \geq 1$. Consider the kernel $G_t$ on the strip $S(t)$, where $q \geq 1$ and $t>0$ are fixed. Suppose $z$ is confined to $[-t,t]$ and $\yh$ is confined to a compact subset $K \subset S(t)$. 

Then as $r=|x| \to \infty$, the kernel decays according to  
\[
G_t(\xh,\yh) =
\begin{cases}
\dfrac{c_{q-1}}{t} \dfrac{1}{r^{q-1}} + O \! \left( \dfrac{1}{r^q} \right) & \text{for $q>1$,} \\[1em]
 \dfrac{1}{t} \log \dfrac{1}{r} + O \! \left( \dfrac{1}{r} \right) & \text{for $q=1$,} 
\end{cases}
\]
and the horizontal and vertical components of the gradient decay according to 
\[
\nabla_{\! x} G_t(\xh,\yh) + \frac{q c_{q+1}}{tr^{q}} \, \vec{e}_r = O\!\left( \frac{1}{r^{q+1}}\right), \qquad 
\frac{\partial\ }{\partial z} G_t(\xh,\yh)=O\!\left( \frac{1}{r^{q+1}} \right) ,
\]
where $\vec{e}_r=x/r$ is the horizontal radial unit vector. The estimates are uniform with respect to $z \in [-t,t]$ and $\yh \in K$.
\end{corollary}
The corollary follows readily from \autoref{pr:twosided}. The constants implicit in the $O(\cdot)$ estimates may depend on $q,t$ and $K$, and the required largeness of $r$ can depend on the set $K$ to which $\yh$ is confined.

\section{\bf Strip kernel: improved asymptotic as $t \to \infty$} \label{sec:improved}

Next we seek finer control of the kernel as the strip thickness tends to infinity, improving on \autoref{pr:kernel}(c). 
\begin{proposition}[Improved asymptotic of the thick strip kernel, $t \to \infty$] \label{pr:larget} Fix $n \geq 1$ and $q \geq 1$. As $t \to \infty$, the kernel $G_t$ on the strip $S(t)$ converges pointwise to the whole-space Riesz kernel, with remainder terms as follows:
\[
G_t(\xh,\yh) = \frac{1}{|\xh-\yh|^q} + \frac{A_q(t)}{(2t)^q}-\frac{B_q(|x-y|^2-(q+1)(z-w)^2) -C_qzw}{(2t)^{q+2}}+O\!\left(\frac{1}{t^{q+3}}\right)
\]
where $A_q(t),B_q,C_q$ were defined in \autoref{th:energybigt}. The remainder term $O(1/t^{q+3})$ is uniform with respect to $\xh$ and $\yh$ provided those points are confined to some bounded subset of $\Rnp$.
\end{proposition}
\begin{proof}[Proof of \autoref{pr:larget}]
First consider $q>1$. Subtracting $1/|\xh-\yh|^q$ and $A_q(t)/(2t)^q$ from the kernel gives that 
\begin{align}
& G_t(\xh,\yh) -\frac{1}{|\xh-\yh|^q}-\frac{1}{(2t)^q}\sum_{j\neq 0}\frac{1}{|j|^q} \nonumber \\ 
& = \sum_{j\neq 0} \left( \frac{1}{(|x-y|^2+(z-2tj-(-1)^jw)^2)^{q/2}} -\frac{1}{|2tj|^q}\right) \nonumber \\
& = \sum_{j\neq 0} \frac{1}{|2tj|^q} D(j), \label{eq:kernelttoinftyexpansion}
\end{align}
say. To estimate this summand $D(j)$, we expand with the binomial series to second order:
\begin{align}
D(j) & = \left(\frac{|x-y|^2+(z-(-1)^jw)^2-4tj(z-(-1)^jw)}{|2tj|^2}+1\right)^{\!\!-q/2}-1 \nonumber \\
& =q\frac{z-(-1)^jw}{2tj}-\frac{q}{2}\frac{|x-y|^2-(q+1)(z-(-1)^jw)^2}{|2tj|^2} +O\!\left( \frac{1}{|tj|^3}\right) \label{eq:dq}
\end{align}
as $t \to \infty$, where the constant implicit in the big $O$ term is uniform with respect to $x,y,z,$ and $w$ since they are bounded. The first term from \eqref{eq:dq} vanishes when summed over $j \neq 0$ in \eqref{eq:kernelttoinftyexpansion}, that is
\[
\sum_{j\neq 0} \frac{1}{|2tj|^q} q \frac{(z-(-1)^jw)}{2tj}=0,
\]
because the summand is odd with respect to $j$.

By plugging the second term from \eqref{eq:dq} into \eqref{eq:kernelttoinftyexpansion} and observing that
\[
(z-(-1)^jw)^2=\begin{cases} (z-w)^2 & \text{if $j$ is even,} \\ (z-w)^2+4zw & \text{if $j$ is odd,}\end{cases}
\]
we obtain that
\begin{align*}
& G_t(\xh,\yh) -\frac{1}{|\xh-\yh|^q}-\frac{1}{(2t)^q}\sum_{j\neq 0}\!\frac{1}{|j|^q} \nonumber \\ 
&= \frac{B_q(-|x-y|^2+(q+1)(z-w)^2)+C_qzw}{(2t)^{q+2}}+O\!\left( \frac{1}{t^{q+3}}\right),
\end{align*}
as needed for the proposition.

Now suppose $q=1$. By definition, the kernel is
\[
G_t(\xh,\yh)=\frac{1}{|\xh-\yh|}+\frac{A_1(t)}{2t} + \sum_{j\neq0} \left(\frac{1}{|\xh-\rho_j(\yh)|}-\frac{1}{|2tj|} \right) .
\]
The conclusion of the proposition for $q=1$ is now obtained by the same binomial expansion calculation as above. 
\end{proof}

\section{\bf Proofs for the strip energy: asymptotics and $t$-derivative}
\label{sec:stripenergy}
 
\subsection*{Proof of \autoref{th:energybigt}} The Riesz $q$-energy $E_K(\infty)$ is finite by hypothesis and so $E_K(t)$ is finite for all $t$ large enough that $K \subset S(t)$, by \autoref{le:finiteness}, and is continuous for those $t$, including at $t=\infty$. The equilibrium measure $\mu_\infty$ is unique since  $E_K(\infty)$ is finite. For $t<\infty$, we denote by $\mu_t$ an equilibrium measure (not necessarily unique) for $E_K(t)$.

Integrating both sides of the kernel expansion from \autoref{pr:larget} with respect to $d\mu_{\infty}(\xh)d\mu_{\infty}(\yh)$ and regarding $\mu_{\infty}$ as a trial measure for $E_{K}(t)$, so that $E_{K}(t)\leq \int_K \! \int_K G_t(\xh,\yh)\, d\mu_{\infty}\, d\mu_{\infty}$, we find  
\[ 
E_{K}(t) - E_K(\infty) - \frac{A_q(t)}{(2t)^{q}} \leq \frac{1}{(2t)^{q+2}}\left( \! -B_q \, M_q(K)+C_q \left( \int_K  z \, d\mu_{\infty}(\xh)\right)^{\!\! 2}\right) + O\!\left(\frac{1}{t^{q+3}}\right) 
\]
as $t \to \infty$.

To obtain an asymptotic inequality in the reverse direction, we integrate both sides of the kernel expansion from \autoref{pr:larget} with respect to $d\mu_t(\xh)d\mu_t(\yh)$ and regard $\mu_t$ as a trial measure for $E_K(\infty)$, so that $E_K(\infty) \leq \int_K \! \int_K |\xh-\yh|^{-q}\, d\mu_t d\mu_t$, hence obtaining that 
\begin{align*}
& E_{K}(t) - E_K(\infty) - \frac{A_q(t)}{(2t)^{q}} \\
& \geq \frac{1}{(2t)^{q+2}}\bigg( \!\! -B_q \int_K\int_K(|x-y|^2-(q+1)(z-w)^2)\,d\mu_{t}(\xh)d\mu_{t}(\yh) \\
& \hspace{8cm} +C_q \left( \int_K  z \, d\mu_{t}(\xh)\right)^{\!\!2}\bigg)+  O\!\left(\frac{1}{t^{q+3}}\right) .
\end{align*}
We will show below that $\mu_t$ converges weak-$*$ to $\mu_\infty$ as $t\to\infty$, which implies that the right side of the preceding inequality equals 
\[ 
\frac{1}{(2t)^{q+2}}\left(\! -B_q \, M_q(K)+C_q \left( \int_K  z \, d\mu_{\infty}(\xh)\right)^{\!\! 2}\right) + o\!\left(\frac{1}{t^{q+2}}\right) .
\]
Combining the upper and lower asymptotic inequalities now yields \autoref{th:energybigt}. 

It remains to show that $\mu_t$ converges weak-$*$ to $\mu_{\infty}$. Let $t_k$ be any sequence converging to $\infty$. Since the $\mu_t$ are probability measures on a fixed set $K$, there is a subsequence $t(l)=t_{k_l}$ such that $\mu_{t(l)}\to \mu$ weak-$*$, for some Borel measure $\mu$ supported on $K$. Weak-$*$ convergence means $\int_K f \, d\mu_{t(l)} \to \int_K f \, d\mu$ as $l \to \infty$, for every continuous function $f$ on $K$. In particular, choosing $f \equiv 1$ gives that $\mu(K)=1$. To prove $\mu=\mu_{\infty}$ (and hence that $\mu_t$ converges weak-$*$ to $\mu_\infty$) it is enough to show  
\begin{equation} \label{eq:Ginfineq}
\int_K \! \int_K |\xh-\yh|^{-q}\, d\mu d\mu \leq E_K(\infty),
\end{equation}
for then uniqueness of the Riesz $q$-equilibrium measure implies that $\mu=\mu_{\infty}$.

Let $N>0$, so that $\min(N,|\xh-\yh|^{-q})$ is a continuous function on $\Rnp \times \Rnp$. Weak-$*$ convergence of $\mu_{t(l)}$ to $\mu$ implies weak-$*$ convergence of the product measures to $\mu \times \mu$ (see \cite[Lemma 6.4]{W81}). Hence 
\begin{align*}
	\int_K \! \int_K \min(N,|\xh-\yh|^{-q})\, d\mu d\mu
	& = \lim_{l\to \infty} \int_K \! \int_K \min(N,|\xh-\yh|^{-q})\, d\mu_{t(l)} d\mu_{t(l)} \\
	& \leq \liminf_{l\to \infty} \int_K \! \int_K G_{t(l)}(\xh,\yh)\, d\mu_{t(l)} d\mu_{t(l)} \\
	& = \liminf_{l\to \infty} E_{K}(t(l)),  \,\, \text{since $\mu_{t(l)}$ is an equilibrium measure} \\
	& = E_K(\infty) 
\end{align*}
by continuity of the energy at $t=\infty$, where the  inequality follows from $|\xh-\yh|^{-q} \leq G_t(\xh,\yh)$ when $q>1$, and when $q=1$ follows from the fact that 
\[
|\xh-\yh|^{-1} = G_t(\xh,\yh) + O(t^{-1} \log t)
\]
by \autoref{pr:larget} with the remainder estimate being uniform over $\xh,\yh \in K$. Now by taking $N\to \infty$ on the left hand side and applying monotone convergence, we obtain the desired inequality \eqref{eq:Ginfineq}.

\subsection*{Proof of \autoref{th:t-to-zero}}
Part (a). Since $K \subset \Rn$, this lower bound on the energy follows immediately from the pointwise kernel inequality in \autoref{pr:kernel}(d). (In fact, that kernel inequality is strict and so the energy bound in part (a) is also strict, provided the right side of the inequality is finite.) 

Part (b). First suppose $1<q<n$. Let $K$ be interior $(q-1)$-capacitable in $\Rn$. We may suppose $V_{q-1}(K)<\infty$, since otherwise there is nothing to prove.
 
Let $\gamma$ be a probability measure on $K$. Using $\gamma$ as a trial measure for $E_{K}(t)$ and multiplying by $t$ gives that
\[
t E_{K}(t) \leq \int_K \! \int_K t G_t(x,y) \, d\gamma d\gamma .
\]
By \autoref{pr:twosided}, we have 
\[
t G_t(x,y) = \dfrac{c_{q-1}}{|x-y|^{q-1}} + \dfrac{1}{|x-y|^q} \, O(t).
\]
If $\int_K \! \int_K |x-y|^{-q} \, d\gamma d\gamma$ is finite (which in general it need not be), then letting $t \to 0$ yields that 
\begin{equation}\label{eq:tto0upperbound}
\limsup_{t \to 0} t E_{K}(t) \leq c_{q-1} \int_K \! \int_K \frac{1}{|x-y|^{q-1}} \, d\gamma d\gamma .
\end{equation}

Suppose $Q \subset K^{int}$ is compact with finite $(q-1)$-energy, where $K^{int}$ is the interior of $K$ as a subset of $\Rn$, and write $\nu$ for the Riesz $(q-1)$-equilibrium measure on $Q$. Consider the mollified measure 
\[
\gamma_\e = \eta_\e * \nu 
\]
where $\eta_\e$ is the uniform measure of mass $1$ on the centered ball of radius $\e$ in $\Rn$, so that $\gamma_\e$ is a unit measure and is supported in $K$ provided $\e$ is sufficiently small, because the compact subset $Q$ where $\nu$ is supported lies at positive distance from the boundary of $K$. The $q$-energy $\int_{K} \! \int_{K}  |x-y|^{-q} \, d\gamma_\e d\gamma_\e$ of $\gamma_\e$ is finite since $q<n$ and 
\[
d\gamma_\e(x) = \frac{1}{|\B^n(\e)|} \nu(x+\B^n(\e)) \, dx
\]
is bounded by $1/|\B^n(\e)|$ times Lebesgue measure on $K$, where the Lebesgue measure itself has finite $q$-energy. Moreover 
\[
\lim_{\e \to 0} \int_{K} \! \int_{K}  \frac{1}{|x-y|^{q-1}} \, d\gamma_\e d\gamma_\e = \int_Q \! \int_Q \frac{1}{|x-y|^{q-1}} \, d\nu d\nu = V_{q-1}(Q)
\]
by \cite[Lemma 1.2 on p.{\,}82]{L72}. Summing up, we know $\gamma_\e$ has finite energy with respect to the Riesz $q$-kernel, and that in the limit as $\e \to 0$ it is energy minimizing on $Q$ with respect to the $(q-1)$-kernel. 

Applying inequality \eqref{eq:tto0upperbound} with the choice $\gamma=\gamma_\e$ on the right side and then letting $\e \to 0$, we obtain that 
\[
\limsup_{t \to 0} t E_{K}(t) \leq c_{q-1} V_{q-1}(Q).
\]
The inequality in part (b) of the theorem now follows by infimizing over all $Q$ on the right side of the inequality and calling on the hypothesis that $K$ is interior $(q-1)$-capacitable.

The proof is essentially the same in the case $q=1$, by using the $q=1$ estimate from \autoref{pr:twosided}, writing $\nu$ for the logarithmic equilibrium measure on $Q$, and relying on \cite[Lemma 1.2{$^{\,\prime}$} on p.{\,}83]{L72} to justify the convergence of the logarithmic energies of the mollified measures, as $\e \to 0$.

Part (c) follows directly from parts (a) and (b).

\subsection*{Proof of \autoref{pr:EKderiv}}
The finiteness claim is immediate from \autoref{le:finiteness}, which shows also that $E_K(t)$ is continuous.

In the proof of that lemma, we decomposed $G_t(\xh,\yh)=|\xh-\yh|^{-q}+H_t(\xh,\yh)$. From the definition of $H_t$ in \eqref{eq:Hdef} one may confirm that $\frac{\partial H_t}{\partial t}(\xh,\yh)$ exists and is continuous jointly as a function of $t>0$, $\xh,\yh \in S(t)$. In that definition, remember that $\rho_j(\yh)$ depends on $t$.

Since $K\subset \Rn,$ we have $\xh=(x,0)$ and $\yh=(y,0)$ whenever $\xh,\yh \in K$, and so slightly abusing notation we may write $G_t(x,y)$ and $H_t(x,y)$. 

Let $\tau>t>0$. Using the equilibrium measure $\mu_t$ (which is unique by \autoref{le:uniqueness}) as a trial measure for $E_K(\tau)$, we have
\begin{align*}
\frac{E_K(\tau)-E_K(t)}{\tau-t} 
& \leq \int_K \! \int_K \frac{G_{\tau}(x,y)-G_t(x,y)}{\tau-t} \, d\mu_t d\mu_t \\
& = \int_K \! \int_K \frac{H_{\tau}(x,y)-H_t(x,y)}{\tau-t} \, d\mu_t d\mu_t.
\end{align*}
Therefore,
\begin{equation} \label{eq:derivupper}
\limsup_{\tau\to t^+} \frac{E_K(\tau)-E_K(t)}{\tau-t} \leq \int_K \! \int_K \frac{\partial H_t}{\partial t}(x,y) \, d\mu_t d\mu_t < \infty,
\end{equation}
since the difference quotient $(H_{\tau}(x,y)-H_t(x,y))/(\tau-t)$ converges uniformly to $\frac{\partial H_t}{\partial t}(x,y)$ on $K \times K$. 

In the other direction, using $\mu_{\tau}$ as a trial measure for $E_K(t)$, we have
\[
\frac{E_K(\tau)-E_K(t)}{\tau-t} 
\geq \int_K \! \int_K \frac{H_{\tau}(x,y)-H_t(x,y)}{\tau-t} \, d\mu_{\tau} d\mu_{\tau} .
\]
Below we will prove that $\mu_{\tau}$ converges weak-$*$ to $\mu_t$ as $\tau \to t$. We therefore have 
\begin{equation} \label{eq:derivlower}
\liminf_{\tau\to t^+} \frac{E_K(\tau)-E_K(t)}{\tau-t} \geq \int_K \! \int_K \frac{\partial H_t}{\partial t}(x,y) \, d\mu_t d\mu_t.
\end{equation} 
Inequalities \eqref{eq:derivupper} and \eqref{eq:derivlower} together prove that the derivative exists from the right, equalling
\[
\lim_{\tau \to t^+} \frac{E_K(\tau)-E_K(t)}{\tau-t} = \int_K \! \int_K \frac{\partial H_t}{\partial t}(x,y) \, d\mu_t d\mu_t.
\]
The argument is similar for the derivative from the left. And clearly $\partial H_t/\partial t = \partial G_t/\partial t$. 

To complete the proof, we now show that $\mu_{\tau}$ converges weak-$*$ to $\mu_t$ as $\tau \to t$. Consider any sequence $t_k$ converging to $t$. There is a subsequence $t(l)=t_{k_l}$ such that $\mu_{t(l)}\to \mu$ weak-$*$ as $l \to \infty$, for some probability measure $\mu$ supported on $K$. We will prove that $\mu=\mu_t$. To do so, since $\mu_t$ is the unique energy-minimizing measure for $E_K(t)$, it is enough to show 
\begin{equation} \label{eq:muismin}
\int_K \! \int_K G_t(x,y) \, d\mu d\mu \leq E_K(t).
\end{equation} 
Fix $N>0$. The function $\min(N, G_t(x,y))$ is continuous, and so by weak-$*$ convergence we have 
\begin{align}
& \int_K \! \int_K\min(N, G_t(x,y))\,  d\mu d\mu \nonumber \\
& =\lim_{l\to \infty} \int_K \! \int_K\min(N, G_t(x,y)) \, d\mu_{t(l)} d\mu_{t(l)} \nonumber \\
& \leq \liminf_{l\to \infty}  \int_K \! \int_K G_t(x,y) \, d\mu_{t(l)} d\mu_{t(l)} \nonumber  \\
& = \liminf_{l\to \infty}  \int_K \! \int_K G_{t(l)}(x,y) \, d\mu_{t(l)} d\mu_{t(l)} \quad \text{as explained below} \label{eq:replacet} \\
& = \liminf_{l\to \infty} E_K(t(l)) = E_K(t) , \nonumber 
\end{align}
where equality \eqref{eq:replacet} follows from the fact that $H_t(x,y)=H_{t(l)}(x,y)+o(1)$ as $l\to \infty$, with the $o(1)$ term being uniform on $K\times K$ by continuity of $H_t$. Letting $N \to \infty$ and applying monotone convergence now proves inequality \eqref{eq:muismin}.

Finally, continuity of $E_K^{\prime}(t)$ follows from continuity of $\frac{\partial H_t}{\partial t}$ and weak-$*$ convergence of $\mu_\tau \to \mu_t$ as $\tau \to t$.

\section{\bf Computable example}\label{sec:computableexample}

To complement the theoretical work in the rest of the paper, in this section we prove an explicit formula for the strip kernel when $n=3$ and $q=2$: 
\begin{equation}\label{eq:kernelspecialcase}
G_t(\xh,\yh) 
= \frac{\pi \sinh \! \frac{\pi |x-y|}{2 t}}{4t|x-y|} \left( \frac{1}{\cosh \! \frac{\pi |x-y|}{2 t} - \cos \! \frac{\pi |z-w|}{2t}} 
+\frac{1}{\cosh \! \frac{\pi |x-y|}{2 t} + \cos \! \frac{\pi  |z+w|}{2t}} \right) 
\end{equation}
for $x,y \in \R^3$ and $z,w \in (-t,t)$. If we further specialize to points in the $3$-plane, with $z=w=0$ so that $\xh=(x,0)$ and $\yh=(y,0)$, then the kernel simplifies easily to 
\begin{equation}\label{eq:kernelsimplified}
G_t(x,y) = \frac{\pi \coth \frac{\pi 
   |x-y|}{2 t}}{2t|x-y|} , \qquad x , y \in \R^3 .
\end{equation}

This formula will not be needed elsewhere in the paper, although it could aid readers who wish to plot the kernel as a function of $|x-y|$ in order to visualize how it behaves in the limiting cases, namely that $G_t(x,y) \to |x-y|^{-2}$ as $t \to \infty$, and $t G_t(x,y) > (\pi/2)|x-y|^{-1}$ with equality as $t \to 0$, and $G_t(x,y) \sim (\pi/2t) |x-y|^{-1}$ as $|x-y| \to \infty$. These special cases are consistent with putting $q=2$ into \autoref{pr:kernel} and \autoref{co:kerneldecay}.  

\begin{proof}[Proof of formula \eqref{eq:kernelspecialcase}] First we show that 
\begin{equation} \label{eq:cleanseries}
\sum_{k \in \Z} \frac{2a}{a^2+(b-k\pi)^2} = \frac{2 \sinh 2a}{\cosh 2a - \cos 2b}  
\end{equation}
when $a \geq 0, -\pi<b<\pi$, and $a$ and $b$ are not both zero. Indeed, the left side of \eqref{eq:cleanseries} equals 
\[
\begin{split}
& \sum_{k \in \Z} \left( \frac{1}{a+bi-k\pi i} + \frac{1}{a-bi+k\pi i} \right) \\
& = \frac{1}{a+bi} + \frac{1}{a-bi} + 2 \sum_{k=1}^\infty \left( \frac{a+bi}{(a+bi)^2+k^2\pi^2} + \frac{a-bi}{(a-bi)^2+k^2\pi^2} \right) 
\end{split}
\]
by splitting off the term with $k=0$ and combining the terms with $\pm k$. A standard series for $\coth z$ (see \cite[(4.36.3)]{DLMF}) reduces the last expression to $\coth(a+bi)+\coth(a-bi)$. Then writing $\coth$ in terms of complex exponentials and combining the two $\coth$'s, one arrives at the right side of \eqref{eq:cleanseries}. 

Applying formula \eqref{eq:cleanseries} with $a=\pi |x-y|/4t$ and $b=\pi(z-w)/4t$ and then multiplying both sides by $\pi/8t|x-y|$ yields that 
\[
\sum_{k \in \Z} \frac{1}{|\xh-\rho_{2k}(\yh)|^2} = \frac{1}{4t|x-y|} \, \frac{\pi \sinh \! \frac{\pi |x-y|}{2 t}}{\cosh \! \frac{\pi |x-y|}{2 t} - \cos \! \frac{\pi (z-w)}{2t}} ,
\]
which shows that the even-indexed terms in $G_t(\xh,\yh)$ sum to give one of the desired expressions in \eqref{eq:kernelspecialcase}. The odd-indexed terms in $G_t(\xh,\yh)$ give the other expression, since applying 
\eqref{eq:cleanseries} with $a=\pi |x-y|/4t$ and $b=\pi(z+w)/4t+\pi/2$ and then multiplying both sides by $\pi/8t|x-y|$ gives 
\[
\sum_{k \in \Z} \frac{1}{|\xh-\rho_{2k-1}(\yh)|^2} =  \frac{1}{4t|x-y|} \, \frac{\pi \sinh \! \frac{\pi |x-y|}{2 t}}{\cosh \! \frac{\pi |x-y|}{2 t} + \cos \! \frac{\pi  (z+w)}{2t}} .
\]
\end{proof}

\noindent \emph{Remarks.} 
1. If $K$ is a compact subset of $\R^3$ then by substituting the kernel formula \eqref{eq:kernelsimplified} above into \autoref{pr:EKderiv}, we find the appealing formula for $(n,q)=(3,2)$ that  
\[
(t E_K(t))^\prime = \frac{\pi^2}{4t^2} \int_K \! \int_K \csch^2 \left( \frac{\pi |x-y|}{2t} \right) d\mu_t d\mu_t , \qquad t>0 .
\]

2. A closed-form expression for $G_t$ can seemingly be obtained whenever the dimension $n$ is odd and $q=n-1$, although the formulas we have found using Mathematica become rapidly complicated as $n$ increases.
\section*{Acknowledgments}
Laugesen's research was supported by grants from the Simons Foundation (\#964018) and the National Science Foundation ({\#}2246537). 

\section*{Statements and Declarations}

\subsection*{Competing interests}
The authors have no financial or non-financial interests that are directly or indirectly related to the work.
The authors have no conflicts of interest to declare that are relevant to the content of this article.
All authors certify that they have no affiliations with or involvement in any organization or entity with any financial interest or non-financial interest in the subject matter or materials discussed in this manuscript.

\subsection*{Data availability}
Not applicable -- no data sets were generated or analysed.

\bibliographystyle{plain}

\end{document}